\documentclass{amsart}
\usepackage[utf8]{inputenc}

\usepackage{mathrsfs}
\usepackage{amscd,amssymb,amsopn,amsmath,amsthm,graphics,amsfonts,enumerate,verbatim,calc}

\newtheorem{theorem}{Theorem}[section]
\newtheorem{claim}[theorem]{Claim}

\newtheorem{lemma}[theorem]{Lemma}

\newtheorem{corollary}[theorem]{Corollary}

\theoremstyle{definition}
\newtheorem{definition}[theorem]{Definition}

\theoremstyle{remark}
\newtheorem{remark}[theorem]{Remark}

\newcommand{\dom}{{\rm dom}}

\newcommand{\bbP}{{\mathbb P}}

\newcommand{\bbQ}{{\mathbb Q}}

\newcommand{\Add}{{\rm Add}}
\newcommand{\cf}{{\rm cf}}

\newcommand{\CH}{{\rm CH}}

\newcommand{\ZFC}{{\rm ZFC}}
\newcommand{\Th}{{\rm Th}}
\newcommand{\Cov}{{\rm Cov}}
\newcommand{\meagre}{{\rm meagre}}
\newcommand{\MA}{{\rm MA}}
\newcommand{\countable}{{\rm countable}}

\newcount\skewfactor
\def\mathunderaccent#1#2 {\let\theaccent#1\skewfactor#2
\mathpalette\putaccentunder}
\def\putaccentunder#1#2{\oalign{$#1#2$\crcr\hidewidth
\vbox to.2ex{\hbox{$#1\skew\skewfactor\theaccent{}$}\vss}\hidewidth}}

\usepackage{hyperref}

\begin{document}

\title {The Keisler-Shelah isomorphism theorem and the continuum hypothesis}

\author[M.  Golshani]{Mohammad Golshani}

\address{Mohammad Golshani, School of Mathematics, Institute for Research in Fundamental Sciences (IPM), P.O.\ Box:
	19395--5746, Tehran, Iran.}

\email{golshani.m@gmail.com}

\author[S. Shelah] {Saharon Shelah}
\address{Einstein Institute of Mathematics\\
Edmond J. Safra Campus, Givat Ram\\
The Hebrew University of Jerusalem\\
Jerusalem, 91904, Israel\\
 and \\
 Department of Mathematics\\
 Hill Center - Busch Campus \\
 Rutgers, The State University of New Jersey \\
 110 Frelinghuysen Road \\
 Piscataway, NJ 08854-8019 USA}
\email{shelah@math.huji.ac.il}
\urladdr{http://shelah.logic.at}
\thanks{ The first author's research has been supported by a grant from IPM (No. 1400030417). The
	second author's research has been partially supported by Israel Science Foundation (ISF) grant no:
	1838/19. This is publication 1215 of second author.}




\begin{abstract}
We show that if for any  two elementary equivalent structures
$\bold M, \bold N$ of size at most continuum in a countable language,  $\bold M^{\omega}/ \mathcal U \simeq \bold N^\omega / \mathcal U$ for some ultrafilter $\mathcal U$ on $\omega,$ then $\CH$ holds. We also provide some consistency results about Keisler and Shelah isomorphism theorems in the absence of
$\CH$.
\end{abstract}

\maketitle
\numberwithin{equation}{section}
\section{introduction}
The Keisler-Shelah isomorphism theorem provides a characterization of elementary equivalence. It says that two models of a theory are elementarily equivalent if and only if they have isomorphic ultrapowers.

Let the \emph{Keisler criterion (for elementary equivalence)} be the statement: for any two structures
$\bold M, \bold N$ of size $\leq 2^{\aleph_0}$ in a countable language, $\bold M \equiv \bold N$ if and only if  $\bold M^{\omega}/ \mathcal U \simeq \bold N^\omega / \mathcal U$ for some ultrafilter $\mathcal U$ on $\omega.$

In \cite{keislr} (see also \cite{keislerphd}), Keisler showed that the Keisler criterion follows from $\CH$.
The result is trivial if at least one of $\bold M, \bold N$ is finite, so assume otherwise. He showed that for any non-principal ultrafilter
$\mathcal U$ on $\omega$ the ultrapowers $\bold M^{\omega}/ \mathcal U$
and $\bold N^{\omega}/ \mathcal U$ are $\aleph_1$-saturated of size $2^{\aleph_0}$. Thus, under $\CH$, both are saturated
of the same size and the result
  follows from the uniqueness of saturated models.

Later Shelah \cite{sh:iso} removed the $\CH$ assumption in Keisler's theorem, by showing that if $\mathcal L$ is a countable language  and $\bold M, \bold N$ are  countable $\mathcal L$-models,
then $\bold M \equiv \bold N$ if and only if there exists an ultrafilter $\mathcal U$ on $2^\omega$ such that  $\bold M^{\omega}/ \mathcal U \simeq \bold N^\omega / \mathcal U.$

 In \cite{sh:vive}, Shelah has constructed a model of $\ZFC$ in which $2^{\aleph_0}=\aleph_2$
and in which there are countable graphs $\bold \Delta \equiv \bold \Gamma$  such that for no ultrafilter $\mathcal U$
on $\omega,$  $\bold \Delta^{\omega}/ \mathcal U \simeq \bold \Gamma^\omega / \mathcal U.$ This shows that
$\CH$ is an essential assumption for Keisler's theorem, even for countable models.

In this paper we discuss some variants of the Keisler's criterion, in particular in the absence of $\CH$, and prove some related results.
First we show that Keisler's criterion  is indeed equivalent to $\CH$ by proving the following theorem.
\begin{theorem}
\label{thm1}
 Suppose $2^{\aleph_0} \geq \aleph_2$. Then the Keisler criterion fails.
\end{theorem}
The counterexample we consider for the above theorem comes from the theory of dense linear orders, see Theorem \ref{thm11}.

It is known from
the work of Ellentuck and Rucker \cite{ellentuck} that if Martin's axiom$+\neg \CH$ holds, then there exists
an ultrafilter $\mathcal U$ on $\omega$ such that for any countable structure $\bold M$, the ultrapower $\bold M^\omega / \mathcal U$
is saturated. In particular if $\bold M \equiv \bold N$ are countable models of the same vocabulary, then
$\bold M^\omega / \mathcal U \simeq \bold N^\omega / \mathcal U.$ We consider models of larger size and ask for the same conclusion.
In particular, we prove the following:
\begin{theorem}
\label{thm2} Suppose $2^{\aleph_0}> \aleph_1=\cf(2^{\aleph_0})$ and $\Cov(\meagre)=2^{\aleph_0}$.
If $\bold M, \bold N$ are models of size $\aleph_1$ in a countable language, then  $\bold M^{\omega}/ \mathcal U \simeq \bold N^\omega / \mathcal U$
for some ultrafilter $\mathcal U$ on $\omega$.
\end{theorem}
We also prove a related  consistency result in the generic extension obtained by adding many Cohen reals, which allows us to remove the cofinality restriction
of the above theorem.

\section{Keisler's theorem and the $\CH$}
In this section we prove the following theorem which immediately implies Theorem \ref{thm1}.
\begin{theorem}
\label{thm11}
There are models $\bold M, \bold N$ of the theory $\Th(\bbQ, <)$ of size $\aleph_0, \aleph_2$ respectively
such that for no ultrafilter $\mathcal U$ on $\omega, $
 $\bold M^{\omega}/ \mathcal U \simeq \bold N^\omega / \mathcal U.$
\end{theorem}
\begin{proof}
 Let $\bold M=(\bbQ, <)$ and let $\bold N$ be a dense linear order of cardinality
$\aleph_2$ such that for some $a, b \in \bold N$ we have $\cf(\bold N_a)=\aleph_1$
and $\cf(\bold N_b)=\aleph_2$, where for each $c \in \bold N,$
\[
\bold N_c=\{ d \in \bold N: d <_{\bold N} c       \}.
\]
We show that $\bold M$ and $\bold N$ are as required. Suppose, towards a contradiction, that for some ultrafilter $\mathcal U$ on $\omega$, there exists an isomorphism   $f: \bold N^{\omega}/ \mathcal U \simeq \bold M^\omega / \mathcal U$. To simplify the notation, let us set
$\bold M_*=\bold M^{\omega}/ \mathcal U \text{~~and~~} \bold N_*=\bold N^\omega / \mathcal U.$ Let
\[
a_* = [\langle a: n<\omega  \rangle]_{\mathcal U} \in \bold N_*
\]
and
\[
b_* = [\langle b: n<\omega  \rangle]_{\mathcal U} \in \bold N_*.
\]
By the choice of elements $a$ and $b$ we have:
\begin{claim}
\label{cl1}
$\cf((\bold N_*)_{a_*})=\aleph_1$ and $\cf((\bold N_*)_{b_*})=\aleph_2.$
\end{claim}
\begin{proof}
Let us show that $\cf((\bold N_*)_{a_*})=\aleph_1$. Pick a $<_{\bold N}$-increasing sequence $\langle  a_i: i<\omega_1       \rangle$
which is $<_{\bold N}$-cofinal in $a$. Then the sequence $\langle (a_i)_*: i<\omega_1        \rangle$, where $(a_i)_*=[\langle a_i: n<\omega      \rangle]_{\mathcal U}$ is an increasing sequence below $a_*.$ Let us show that it is also cofinal in $a_*.$
Thus let $f: \omega \to \bold N$ and suppose that $[f]_{\mathcal U} <_{N_*} a_*.$ Without loss of generality
$f: \omega \to \bold N_{a}$. For each $n<\omega$ pick some $i(n)< \omega_1$ such that $f(n) <_{\bold N} a_{i(n)}$. Let $j=\sup_{n \to \infty}i(n).$
Then $j < \omega_1$ and for every $n<\omega$
\[
f(n) <_{\bold N} a_{i(n)} <_{\bold N} a_j,
\]
in particular $[f]_{\mathcal U} <_{\bold N_*} (a_j)_*$. Thus the sequence  $\langle (a_i)_*: i<\omega_1        \rangle$ is increasing and cofinal in
$a_*$. By the regularity of $\aleph_1,$ we have $\cf((\bold N_*)_{a_*})=\aleph_1$.
\end{proof}
Set $a_\dagger=f(a_*)$ and $b_\dagger=f(b_*)$.
\begin{claim}
\label{cl2}
$\cf((\bold M_*)_{a_\dagger})=\aleph_1$ and $\cf((\bold M_*)_{b_\dagger})=\aleph_2.$
\end{claim}
\begin{proof}
It is trivial by the choice of $a_\dagger$ and $b_\dagger$.
\end{proof}
\begin{claim}
\label{cl3}
There is a function $F: \bold M^3 \to \bold M$ such that for every $c, d \in \bold M$, the formula $F(x, c, d)$ defines an automorphism of $\bold M$
which maps $c$ to $d$.
\end{claim}
\begin{proof}
Define $F$ by $F(x, y, z)=x-y+z.$ The function $F$ is easily seen to be as required.
\end{proof}
It follows from Claim \ref{cl3} that for some function $F_*,$ $(\bold M_*, F_*)= (\bold M, F)^{\omega} / \mathcal U.$ Then by the choice of $F$, the function $F_*$ has the following property:
\begin{enumerate}
\item[($\ast$):] $F_*: \bold M_*^3 \to \bold M_*$ is a function such that for all $c, d \in \bold M_*$, the formula $F_*(x, c, d)$ defines an automorphism of $\bold M_*$
which maps $c$ to $d$.
\end{enumerate}
In particular $F_*(x, a_\dagger, b_\dagger )$ defines an automorphism of $\bold M_*$
which maps $a_\dagger$ to $b_\dagger$. Thus we must have
\[
\cf((\bold M_*)_{a_\dagger})=\cf((\bold M_*)_{b_\dagger}),
\]
which contradicts Claim \ref{cl2}.
\end{proof}
By the above result and Keisler's theorem, we have the following corollary.
\begin{corollary}
\label{keqch}
The following are equivalent:
\begin{enumerate}
\item[(a)] $\CH$,

\item[(b)] Keisler's criterion: if $\mathcal L$ is a countable language  and $\bold M, \bold N$ are  $\mathcal L$-models of size $\leq 2^{\aleph_0}$,
then $\bold M \equiv \bold N$ if and only if there exists an ultrafilter $\mathcal U$ on $\omega$ such that  $\bold M^{\omega}/ \mathcal U \simeq \bold N^\omega / \mathcal U.$
\end{enumerate}
\end{corollary}

\section{Keisler-Shelah theorem for models of cardinality $\aleph_1$}
\label{s2}
In this section, we ask to what extent the Keisler and Shelah isomorphism theorems can hold
 for models of uncountable cardinality. We prove some theorems  which by the result of the previous section are, in some sense, optimal.
  \begin{definition}
 \begin{enumerate}
\item Let $\Cov(\meagre)$ be the minimal size of a family of meagre subsets of the real line that cover it.

\item Given an infinite cardinal $\kappa$, let $\MA_\kappa(\countable)$ be the following statement: if $\mathbb{P}$ is a countable partial order
 and $\mathcal A$ is a family of dense subsets of $\mathbb{P}$ of size $\kappa$, then there exists a filter $\bold G \subseteq \mathbb{P}$
 meeting all sets in $\mathcal A.$
 \end{enumerate}
 \end{definition}
 Our proof relies on the following lemma.
 \begin{lemma}
 \label{covvsma}
 (see \cite[Theorem 7.13]{blass}) Suppose $\Cov(\meagre)=2^{\aleph_0}$. Then $\MA_{\kappa}(\countable)$ holds for all $\kappa < 2^{\aleph_0}$.
 \end{lemma}

 Let us start by proving Theorem \ref{thm2}.
\begin{theorem}
\label{thm3} Suppose $2^{\aleph_0}> \aleph_1=\cf(2^{\aleph_0})$ and $\Cov(\meagre)=2^{\aleph_0}$. Suppose $\bold M_0 \equiv \bold M_1$
are models of size $\leq \aleph_1$ in the same countable vocabulary $\mathcal L$. Then for some ultrafilter $\mathcal U$ on $\omega,$
$\bold M_0^{\omega}/ \mathcal U \simeq \bold M_1^\omega / \mathcal U.$
\end{theorem}
\begin{proof}
Before giving the details of the proof, let us sketch the main idea. We would like to find
an ultrafilter $\mathcal U$ on $\omega$ and enumerations $\langle g^0_\alpha: \alpha < 2^{\aleph_0}    \rangle$
and $ \langle  g^1_\alpha: \alpha< 2^{\aleph_0}          \rangle$ of $\bold M_0^\omega$ and $\bold M_1^\omega$ respectively, such that
\begin{equation}
\label{eq1}
((\bold M_0)^\omega/\mathcal U, [g^0_0]_{\mathcal U}, \cdots, [g^0_\alpha]_{\mathcal U}, \cdots  ) \equiv ((\bold M_1)^\omega/\mathcal U, [g^1_0]_{\mathcal U}, \cdots, [g^1_\alpha]_{\mathcal U}, \cdots  ).
\end{equation}
This will show that the function $\langle  ([g^0_\alpha]_{\mathcal U}, [g^1_\alpha]_{\mathcal U}): \alpha < 2^{\aleph_0}                 \rangle$
is an isomorphism between $\bold M_0^{\omega}/ \mathcal U$ and $\bold M_1^\omega / \mathcal U.$
On the other hand, by Lo\'{s} theorem,  \ref{eq1} is equivalent to saying that for all $\mathcal L$-formula $\phi(x_0, \cdots, x_{n-1})$
and all $\beta_0, \cdots, \beta_{n-1}< 2^{\aleph_0}$,
\[
\bigg\{k<\omega: \bold M_0 \models \phi(g^0_{\beta_0}(k), \cdots, g^0_{\beta_{n-1}}(k)) \Leftrightarrow \bold M_1 \models \phi(g^1_{\beta_0}(k), \cdots, g^1_{\beta_{n-1}}(k))\bigg\} \in \mathcal U.
\]
We define by induction on $\alpha < 2^{\aleph_0}$, a sequence $\langle (\mathcal U_\alpha, g^0_\alpha, g^1_\alpha): \alpha < 2^{\aleph_0} \rangle$,
where $\langle \mathcal U_\alpha: \alpha < 2^{\aleph_0} \rangle$ is an increasing and continuous chain of filters on
$\omega$ such that \ref{eq1} holds whenever $\mathcal U$ replaced by $\mathcal U_{\alpha+1}$. To make sure that $g^0_\alpha$'s and $g^1_\alpha$'s enumerate
all elements of $\bold M_0^\omega$ and $\bold M_1^\omega$ respectively, we use a back and forth construction. To make sure that the construction
continues to work at all levels below $2^{\aleph_0}$, we use the assumption $\Cov(\meagre)=2^{\aleph_0}$ and proceed in such a way that $\mathcal U_\alpha$ is generated by $\leq \aleph_0+|\alpha|$ many elements.

Let us now go into the details of the proof.
Let $\langle \lambda_i: i<\omega_1  \rangle$
be an increasing and continuous sequence of cardinals $\geq \aleph_1$ cofinal in $2^{\aleph_0}$ and for $\ell <2$ let $\langle \bold M^\ell_i: i<\omega_1  \rangle$
be an increasing and continuous chain of elementary submodels of $\bold M_\ell$ such that for all $i<\omega_1, ||\bold M^\ell_i|| = \aleph_0$
and $\bold M_\ell = \bigcup\limits_{i<\omega_1}\bold M^\ell_i$.
Let
$\langle  f^\ell_{\alpha}: \alpha < 2^{\aleph_0}    \rangle$
be an enumeration of $\bold M_\ell^{\omega}$ such that
\[
\alpha < \lambda_i \implies f^\ell_{\alpha} \in ( \bold M^{\ell}_{i})^{\omega}.
\]
Let also
$\langle X_\alpha: \alpha < 2^{\aleph_0}   \rangle$ enumerate $\mathcal P(\omega)$.
By induction on $\alpha < 2^{\aleph_0}$ and using a back and forth construction, we build the triple $(\mathcal U_\alpha, g^0_{\alpha},
g^1_{ \alpha})$
such that:
\begin{enumerate}
\item[(a)] $g^0_{\alpha} \in \bold M_0^\omega$,
furthermore if $\alpha < \lambda_i$, then $g^0_{\alpha} \in (\bold M^0_i)^\omega$,

\item[(b)] $g^1_{\alpha} \in \bold M_1^\omega$,
furthermore if $\alpha < \lambda_i$, then $g^1_{\alpha} \in (\bold M^1_i)^\omega$,

\item[(c)] for $i<\omega_1$ and $\ell < 2$, $\{g^\ell_{\alpha}: \alpha < \lambda_i   \}= \{f^\ell_{\alpha}: \alpha < \lambda_i   \},$

\item[(d)] $\mathcal U_\alpha$ is a filter on $\omega$ generated by $\leq \aleph_0+|\alpha|$ sets containing all co-finite subsets of $\omega,$

\item[(e)] if $\phi(x_0, \cdots, x_{n-1})$ is a formula of $\mathcal L$ and $\beta_0, \cdots, \beta_{n-1} \leq  \alpha,$ then the set
$Y_{\phi, \langle \beta_0, \cdots, \beta_{n-1}\rangle}$ defined as
\[
\bigg\{ k< \omega: \bold M_0 \models \phi(g^0_{\beta_0}(k), \cdots, g^0_{\beta_{n-1}}(k)) \Leftrightarrow
\bold M_1 \models \phi(g^1_{\beta_0}(k), \cdots, g^1_{\beta_{n-1}}(k))                         \bigg\},
\]
belongs to $\mathcal U_{\alpha+1},$

\item[(f)] if $\alpha < \beta< 2^{\aleph_0},$ then $\mathcal U_\alpha \subseteq \mathcal U_\beta$.

\item[(g)] if $\alpha$ is a limit ordinal, then $\mathcal U_\alpha=\bigcup\limits_{\beta<\alpha}\mathcal U_\beta$,

\item[(h)] for all $\alpha < 2^{\aleph_0},$ either $X_\alpha \in \mathcal U_{\alpha+1}$ or $\omega\setminus X_\alpha \in \mathcal U_{\alpha+1}$.
\end{enumerate}
As in Shelah \cite[\text{Ch} VI, \S 3]{sh:c}, there is no problem in carrying the induction, however let us elaborate the main point
of the proof. The only difficulty in carrying the induction is clause (e).
Thus suppose that $\alpha< 2^{\aleph_0}$ and the construction is done up to $\alpha.$ Let also $i<\omega_1$ be such that
$\alpha < \lambda_i.$
First suppose that $\alpha$ is an even ordinal. Let $g^0_\alpha=f^0_{\gamma_\alpha}$,
where $\gamma_\alpha$ is the least ordinal such that $f^0_{\gamma_\alpha} \notin \{g^0_\beta: \beta < \alpha    \}.$
Note that $\gamma_\alpha < \lambda_i.$ Let also $\mathscr{G}(\mathcal U_\alpha)$
be a set of generators of $\mathcal U_\alpha$ of size $\leq \aleph_0+|\alpha|$.

Let $\bbP$ be the forcing notion consisting of all maps $p: \dom(p) \to M^1_i$, where $\dom(p)$ is a finite subset of $\omega,$ ordered by inclusion. $\bbP$
is countable. Define the following  sets:
\begin{itemize}
\item $D_n=\{p \in \bbP: n \in  \dom(p)         \}$, where $n<\omega$.
\item For any set $A \in \mathscr{G}(\mathcal U_\alpha)$, any finite sequence  $\vec \phi=\langle \phi_\iota(x_0, \cdots, x_{n_\iota-1}, y): \iota \in I \rangle$, any finite sequence $\langle \vec{\beta_\iota^\ell}: \iota \in I, \ell \in J_\iota \rangle$, where $\vec{\beta_\iota^\ell}=\langle \beta^\ell_{\iota,0}, \cdots, \beta^\ell_{\iota, n_\iota-1} \rangle$ consists of ordinals less than $\alpha$ and $m<\omega$ let $\Sigma_{A, \vec \phi, \langle \vec{\beta_\iota^\ell}: \iota \in I, \ell \in J_\iota\rangle, m}$ be the set of all conditions $p \in \bbP$ such that for some $k \in \dom(p) \cap A$ with $k> m$
    and all $\iota \in I$ and $\ell \in J_\iota:$
\[
\bold M^0_i \models \phi_\iota(g^0_{\beta^\ell_{\iota,0}}(k), \cdots, g^0_{\beta^\ell_{\iota, n_\iota-1}}(k), g^0_\alpha(k)) \Leftrightarrow \bold M^1_i \models \phi_\iota(g^1_{\beta^\ell_{\iota,0}}(k), \cdots, g^1_{\beta^\ell_{\iota, n_\iota-1}}(k), p(k)).
\]
\end{itemize}
Let us show that each of the sets defined above  are dense in $\bbP$. This is clear for the sets $D_n$. Now suppose that
$A, \vec \phi, $  $\langle \vec{\beta_\iota^\ell}: \iota \in I, \ell \in J_\iota \rangle$ and $m$ are given as above
and suppose that $p \in \bbP$. We find some $q\supseteq p$ in $\Sigma_{A, \vec \phi, \langle \vec{\beta_\iota^\ell}: \iota \in I, \ell \in J_\iota\rangle, m}$. Let $k > m, \max(\dom(p))$ be such that $k \in A$. Such a $k$ exists as $A$ is unbounded in $\omega.$ Now
$x=g^0_\alpha(k)$ witnesses
\[
\bold M^0_i \models \exists x \bigwedge_{\iota \in I, \ell \in J_\iota}\phi_\iota(g^0_{\beta^\ell_{\iota,0}}(k), \cdots, g^0_{\beta^\ell_{\iota, n_\iota-1}}(k),x),
\]
and hence by our induction hypothesis  we can find some $y \in \bold M^1_i$ such that
\[
M^1_i \models \bigwedge_{\iota \in I, \ell \in J_\iota} \phi_\iota(g^1_{\beta^\ell_{\iota,0}}(k), \cdots, g^1_{\beta^\ell_{\iota, n_\iota-1}}(k), y).
\]
Set $q=p \cup \{(k, y)\}$. Then $q \in \bbP$ is as required.

The number of the sets we defined above is at most
\[
\aleph_0 + |\mathscr{G}(\mathcal U_\alpha)| \cdot \aleph_0 \cdot |\alpha|^{<\aleph_0} \cdot \aleph_0 =\aleph_0+|\alpha|
\]
which is less than $\lambda_i$, and hence as $\MA_{\lambda_i}(\countable)$
holds, there exists a filter $\bold G \subseteq \bbP$ meeting all the above dense sets. Set $g^1_\alpha=\bigcup_{p \in \bold G}p$. Let
$\mathcal U'_{\alpha+1}$ be the filter generated by
\begin{center}
$\mathcal U_\alpha \cup \{Y_{\phi, \langle \beta_0, \cdots, \beta_{n-1}  \rangle}: \phi, \beta_0, \cdots, \beta_{n-1}$ as in clause (e)$\}$.
 \end{center}
 By the choice of the sets
$\Sigma_{A, \vec \phi, \langle \vec{\beta_\iota^\ell}: \iota \in I, \ell \in J_\iota\rangle, m}$, the above set
 has the finite intersection property and hence $\mathcal U'_{\alpha+1}$ is a proper filter. Now let $\mathcal{U}_{\alpha+1}$
 be the filter generated by $\mathcal U'_{\alpha+1} \cup \{X_\alpha\}$ if this is a proper filter and let
 $\mathcal{U}_{\alpha+1}$
 be the filter generated by $\mathcal U'_{\alpha+1} \cup \{\omega \setminus X_\alpha\}$ otherwise.
If $\alpha$ is an odd ordinal, proceed in the same way, changing the role of the indices $0$ and $1$.

This completes the induction construction.
Set
\[
\mathcal U=\bigcup\{ \mathcal U_\alpha: \alpha < 2^{\aleph_0}          \}.
\]
Then $\mathcal U$ is a non-principal ultrafilter on $\omega$ and $\bold M_0^{\omega}/ \mathcal U \simeq \bold M_1^\omega / \mathcal U$
as witnessed by the function
\[
\langle   ([g^0_{\alpha}]_{\mathcal U}, [g^1_{\alpha}]_{\mathcal U}) : \alpha < 2^{\aleph_0}                \rangle.
\]
This completes the proof of the theorem.
\end{proof}
We close the paper by proving the following consistency result, which is an analogue of Theorem \ref{thm2}, but the cofinality restriction
on $2^{\aleph_0}$ is removed.

 Let us recall that the Cohen forcing  $\Add(\omega, \lambda)$
 for adding $\lambda$ many new Cohen reals is defined as
 $\Add(\omega, \lambda)=\{p: p$ is a finite partial function from $\omega \times \lambda$ into $2     \}$,
 ordered by inclusion.
 \begin{remark} (see \cite[Proposition 22.10]{lorenz})
  $\Add(\omega, \lambda)$  forces $ \Cov(\meagre)=2^{\aleph_0}$.
 \end{remark}
\begin{theorem}
\label{lm1}
Suppose $\lambda> \aleph_1$ and $\lambda^{\aleph_0}=\lambda.$ Let $\bbP=\Add(\omega, \lambda)$. Then in $V[\bold G_{\bbP}]$, the following holds:
if $\bold M_0 \equiv \bold M_1$
are models of size $\leq \aleph_1$ of the same countable vocabulary $\mathcal L$, then for some ultrafilter $\mathcal U$ on $\omega,$
$\bold M_0^{\omega}/ \mathcal U \simeq \bold M_1^\omega / \mathcal U.$

\end{theorem}
\begin{proof}
We may assume that $\cf(\lambda)> \aleph_1$, as otherwise the result follows from Theorem \ref{thm2}.
Now suppose that $\bold M_0 \equiv \bold M_1$ are models of size $\leq \aleph_1$ of a countable vocabulary in $V[\bold G_\bbP]$.
Then for some $\bar\lambda < \lambda,\ \bold M_0, \bold M_1 \in V[\bold G_{\bbP \restriction \bar\lambda}].$
By replacing $V$ by $V[\bold G_{\bbP \restriction \bar\lambda}]$,
we may assume that $\bold M_0, \bold M_1 \in V.$

As $|\lambda\cdot\omega_1|=\lambda$, we may assume that $\bbP$ is $\Add(\omega, \lambda\cdot\omega_1)$ so that forcing with $\bbP$
adds a sequence $\langle r_{\alpha, i}: \alpha < \lambda, i<\omega_1      \rangle$
of reals of order type $\lambda \cdot\omega_1$.

For $i<\omega_1$, set $\bbP_i=\Add(\omega, \lambda \cdot i)$. As $\bbP$ is c.c.c., for every $X \subseteq \omega, X \in V[\bold G_{\bbP}]$, there exists some $i<\omega_1$ such that $X \in V[\bold G_{\bbP_i}]$.
Proceed as in the proof of
 Theorem \ref{thm2} with:
 \begin{itemize}
 \item $\lambda_i=\lambda \cdot (1+i)$,
 \item $\langle \bold M^\ell_i: i<\omega_1 \rangle$ as there,
 \item $\langle f^\ell_{\alpha}: \alpha < 2^{\aleph_0} \rangle$ is an enumeration of $\bold M_\ell^\omega$ in such a way that
 for
 $\alpha < \lambda \cdot(1+i)$, $f^\ell_{\alpha} \in \bold M^\ell_i.$
 \end{itemize}
The rest of the argument is essentially as before.
\end{proof}
\subsection*{Acknowledgements} The authors thank the referee of the paper for his/her very  helpful comments and suggestions.

\end{document}